\documentclass[11pt]{amsart}

\usepackage{graphicx}

\newtheorem{thm}{Theorem}[section]

\newtheorem{lem}[thm]{Lemma}
\newtheorem{prop}[thm]{Proposition}

\newtheorem{quest}[thm]{Question}

\theoremstyle{definition}

\theoremstyle{remark}

\newtheorem{remk}[thm]{Remark}

\renewcommand{\phi}{\varphi}

\newcommand{\tM}{\widetilde{M}}

\newcommand{\tg}{\widetilde{g}}

\newcommand{\tf}{\widetilde{f}}

\newcommand{\area}{{\rm area}}

\newcommand{\Vol}{{\rm Vol}}

\newcommand{\Hbb}{ {\mathbb H}}

\newcommand{\Rbb}{{\mathbb R}}

\newcommand{\Nbb}{{\mathbb N}}

\newcommand{\Hbf}{{\mathbf H}}

\newcommand{\Isom}{{\rm Isom}}

\newcommand{\be}{ \begin{equation} }
\newcommand{\ee}{ \end{equation} }
\begin{document}
\title[Geodesics and nodal sets]{
Geodesics and nodal sets of Laplace eigenfunctions on hyperbolic manifolds}

\author{Chris Judge}
\thanks{The work of C. J. was supported in part by a Simons Collaboration Grant.}
\address{Indiana University, Bloomington}
\email{cjudge@indiana.edu}
\author{Sugata Mondal}
\address{Rawles Hall, Indiana University, Bloomington}
\email{sumondal@iu.edu}

\date{\today}

\maketitle

\begin{abstract}
Let $X$ be a manifold equipped with a complete Riemannian metric of constant
negative curvature and finite volume. We demonstrate the finiteness of the
collection of totally geodesic immersed hypersurfaces in $X$ 
that lie in the zero-level set of some Laplace eigenfunction. 
For surfaces, we show that the number can be 
bounded just in terms of the area of the surface. We also provide constructions
of geodesics in hyperbolic surfaces that lie in a nodal set but that do not lie
in the fixed point set of a reflection symmetry.
\end{abstract}

\section{Introduction}

The structure of the zero level set---or {\em nodal set}---of eigenfunctions of the
Laplacian on Riemannian manifolds is a topic of continuing interest \cite{Zelditch13}.
Among all manifolds, the surfaces of constant negative curvature remain an 
important testing ground \cite{GRS13}. Using a lemma of \cite{TthZld09}, Jung  
\cite{Jung14} showed that each closed curve of nonzero constant geodesic curvature 
intersects each nodal set in a finite set, and, moreover, estimated the size of 
the intersection in terms of the eigenvalue.\footnote{Only closed horocycles 
and boundaries of disks were considered in \cite{Jung14}, 
but the methods also apply to curves
consisting of points that are equidistant to a closed geodesic.} 

On the other hand, it is well-known that closed curves of zero geodesic 
curvature---geodesics---may lie in the zero level set of an eigenfunction. 
Indeed, if a closed Riemannian manifold admits a reflection self-isometry $r$, 
then the fixed point set of $r$ is totally geodesic and infinitely many 
eigenfunctions vanish on the fixed point set.  

Let $N(X)$ denote the cardinality of the collection of totally geodesic immersed 
hypersurfaces $M$ in $X$ such that there exists a Laplace eigenfunction  $\phi$ 
so that either $\phi|_M= 0$ or $\nu \phi =0$ for each vector $\nu$ normal to $M$.
In this note we prove

\begin{thm} \label{thm:main}
If $X$ is a finite volume hyperbolic manifold, then $N(X)$ is finite.   
\end{thm}

In contrast, if $X$ is the standard sphere or flat torus, then 
$N(X)$ is the first uncountable ordinal. Indeed, if the isometry group of a 
compact Riemannian manifold contains a one parameter subgroup $K$
and a reflection that does not commute with $K$, then for each $k \in K$
the fixed point set of $k \circ r \circ k^{-1}$ is a totally geodesic
hypersurface that belongs to the nodal set of an eigenfunction. 

In \S \ref{section:main-thm-proof}, we prove Theorem \ref{thm:main}.
We use the Schwarz reflection principle to show that $N(X)$ is bounded above 
by the number, $N'(X)$, of totally geodesic 
immersed hypersurfaces $M$ such that there exists an orbifold covering 
$p: X \to Y$ such that $p(M)$ lies in the boundary of $Y$. 
We then use a theorem of H.-C. Wang \cite{Wang89} and a compactness argument
to show that $N'(X)$ is finite. 
In \S \ref{section:surface}, we address the question as to whether one can bound 
$N(X)$ in terms of the volume of $X$. We answer affirmatively in the 
case where $X$ is two-dimensional. 
In \S \ref{ex:figure-eight}, we construct some nodal sets on hyperbolic
surfaces that do not belong to the fixed point set of a reflection isometry. 
In \S \ref{section:questions}, we raise some questions.

We thank Kevin Pilgrim and Dylan Thurston for enlightening discussion.
We also thank the anonymous referee for very helpful criticism and especially 
for bringing \cite{Wang89} to our attention.


\section{The proof of Theorem \ref{thm:main}}  \label{section:main-thm-proof}

Recall that a hyperbolic manifold is a quotient of hyperbolic space 
$\Hbf^n$ by a torsion-free discrete subgroup of the Lie group, $\Isom(\Hbf^n)$, 
self-isometries of $\Hbf^n$. More generally, we will consider quotients by 
discrete subgroups of $\Isom(\Hbf^n)$ that include finite order 
elements including reflections. 

The group $\Isom(\Hbf^n)$ acts on functions
$f: \Hbf^n \to \Rbb$ by precomposition: $(f,g) \mapsto f \circ g$.
Let $\Gamma_f \subset \Isom(\Hbf^n)$ denote
the set of $g \in \Isom(\Hbf^n)$ so that either $f \circ g =f$ or $f \circ g = -f$.
If $f$ is continuous, then $\Gamma_f$ is closed.

\begin{lem}  \label{lem:reflection}
If $f$ is a nonconstant continuous function such that
$\Gamma_f$ contains a cofinite discrete subgroup $\Gamma \subset \Isom(\Hbf^n)$,
then $\Gamma_f$ is also discrete.
\end{lem}

\begin{proof}
Suppose to the contrary that $\Gamma_f$ is not discrete. Then since it is a closed
subgroup of a Lie group, it contains a one-parameter subgroup $G$.
Since $\Gamma$ is cofinite, its limit set in the ideal boundary of $\Hbf^n$ is dense.
It follows that the group generated by the conjugates $\gamma G \gamma^{-1} \neq G$, 
$\gamma \in \Gamma$, act transitively on $\Hbf^n$. In particular, $\Gamma_f$ acts 
transitively  and so $f$ is constant.
\end{proof}

\begin{remk}
Lemma \ref{lem:reflection} is similar 
to an observation made about Laplace eigenfunctions in the last paragraph of section 2.2 of \cite{GRS13}.
\end{remk}

If $f: \Hbf^n \to \Rbb$ is an eigenfunction of the Laplacian that
vanishes on a totally geodesic hypersurface $M \subset \Hbf^n$, 
then the reflection $r_{M} \in \Isom(\Hbf^n)$
about $M$ belongs to $\Gamma_f$. This is a well-known variant
of the Schwarz reflection principle. Here is a sketch of the proof:
Since the Laplacian commutes with the action by isometries, the 
function $f \circ r_{M}$
is an eigenfunction of $\Delta$ with the same eigenvalue. Define $h: \Hbf^n \to \Rbb$
by setting $h$ equal to $f$ on one component of $\Hbf^n - M$ and equal to
$-f \circ r_M$ on the other component.
Because both $f$ and $-f \circ r_M$ vanish on $M$,
the function $h$ belongs to the first Sobolev space and $h$ is a weak solution to
the Laplace eigenvalue problem. A well-known argument using elliptic regularity
gives that $h$ is a strong solution. But $f$ and $h$ agree on an open set
and hence $f = h$ and $f= -f \circ r_M$.  A similar argument shows that if 
the normal derivative of an eigenfunction  $f$ vanishes on $M$, then $f \circ r_M = f$.

Given a totally geodesic hypersurface $M \subset \Hbf^n/\Gamma$,
let $\widetilde{M}$ denote a lift to $\Hbf^n$. Since each lift differs by an element 
of $\Gamma$, the group $\Gamma(M)$ generated by $\Gamma$ and $r_{\widetilde{M}}$ 
does not depend on the lift of $M$.   

Suppose that a nonconstant eigenfunction $f$ (or its normal derivative) vanishes on a 
totally geodesic immersed hypersurface $M$. 
Since its lift, $\tf$, to $\Hbb^n$ is nonconstant,  Lemma \ref{lem:reflection} 
implies that $\Gamma_{\tf}$ is discrete and hence so is $\Gamma(M)$. 
Therefore, $N(X)$ is at most the number, $N'(X)$, of totally geodesic 
immersed hypersurfaces $M$ such that $\Gamma(M)$ is discrete.

In order to show that $N'(X)$ is finite, we first note that by a theorem of Wang \cite{Wang89},
the cofinite group $\Gamma$ is contained in at most finitely many discrete subgroups of $\Isom(\Hbb^n)$.
Thus, the number of possible discrete groups of the form $\Gamma(M)$ is finite. 
Thus, the following proposition completes the proof of Theorem \ref{thm:main}.

\begin{prop} \label{prop:finite-realized}

If $\Gamma$ is a cofinite discrete subgroup of $\Isom(\Hbf^n)$ and $\Gamma^* \supset \Gamma$ is discrete, 
then only finitely many totally geodesic immersed hypersurfaces $M \subset \Hbf^n/\Gamma$ 
satisfy $\Gamma(M) = \Gamma^*$.  
\end{prop}

\begin{proof}
First we note that there is a compact subset $K \subset \Hbf^n/\Gamma$ such that 
each totally geodesic immersed submanifold $M \subset  \Hbf^n/\Gamma$ must intersect $K$.  
Indeed, let $K$ be the complement of a union of disjoint horoball neighborhoods of the cusps.
Any lift to $\Hbf^n$ of a geodesic on $M$ joins two distinct points on the ideal boundary.
If an endpoint does lie at the end of a cusp, then it enters $K$ and if 
both endpoints lie in the closure of horoballs, they lie in distinct horoballs,
and hence must enter $K$. 

Suppose that $M_k$ is an infinite sequence of distinct 
totally geodesic immersed hypersurfaces
with $\Gamma(M_k)= \Gamma^*$ for each $k$. Since each $M_k$ intersects $K$, 
there exists an orthonormal frame $F_n$ above some point in $K$ such that 
the first $n-1$ vectors in the frame belong to $TM$.  Since $K$ is compact, so is the
frame bundle over $K$, and hence we may extract a subsequence converging to a frame 
$F_{\infty}$. 
Exponentiation of the space spanned by the first $n-1$ vectors in $F_{\infty}$ determines 
a totally geodesic immersed hypersurface $M_{\infty}$. Lift $M_{\infty}$ and  the $M_n$ 
to totally geodesic hypersurfaces $\widetilde{M}_{\infty}$ and $\widetilde{M}_k$ in $\Hbf^n$
so that  $\widetilde{M}_{k}$ converges to $\widetilde{M}_{\infty}$ uniformly on compact 
subsets. It follows that the sequence of distinct elements $r_{\tM_k}$ converges
to $r_{\tM_{\infty}}$. This contradicts the discreteness of $\Gamma^*$.  
\end{proof}


\section{Towards bounding  $N(X)$  in terms of of the volume of $X$} \label{section:surface}

It is natural to ask whether $N(X)$ can be bounded in terms of the volume of $X$. 
As seen in the previous section, the finiteness of $N'(X)$---and hence $N(X)$---is obtained 
combining the main theorem of \cite{Wang89} and Proposition \ref{prop:finite-realized}. 
As far as the authors are aware, the
main theorem of \cite{Wang89} has not yet been made quantitative, 
and, moreover, we do not know how to quantify Proposition \ref{prop:finite-realized} 
for general $n$. Nonetheless, in this section, we will show that the number of distinct $\Gamma(M)$ 
can be bounded in terms of $\Vol(M)$ and also that Proposition \ref{prop:finite-realized} can
be quantified if $n=2$. In particular, we will obtain

\begin{thm} \label{thm:surface}
For each $A>0$, there exists $C_A$ such that if $X$ is a 
 hyperbolic surface with $\area(X)<A$, then $N(X) \leq N'(X) < C_A$.
\end{thm}

We begin with a bound of the number, $N^*(X)$, of distinct discrete groups $\Gamma(M)$
that arise from totally geodesic immersed hypersufaces $M \subset X$.

\begin{thm} \label{thm:discrete-finite}
For each $V>0$
there exists a constant $C_V$ such that if $X$ is an $n$-dimensional hyperbolic manifold
with volume less than $V$, then $N^*(X) < C_V$.
\end{thm}

\begin{proof}
Let $X= \Hbb/\Gamma$, let $r$ be the reflection in a lift of a totally
geodesic immersed hypersurface $M$, and suppose that $\Gamma(M)=\langle \Gamma, r\rangle$
is discrete. Both $\Gamma$ and $r \Gamma r^{-1}$ are finite index in 
$\Gamma(M) =\langle \Gamma, r\rangle$, and hence $\Gamma \cap r\Gamma r^{-1}$
is finite index in both $\Gamma$ and $r \Gamma r^{-1}$. In other words, 
$r$ lies in the commensurator, ${\rm Comm}(\Gamma)$, of $\Gamma$
\cite{Maclachlan-Reid}. In particular, $\Gamma(M)$ is a subgroup of ${\rm Comm}(\Gamma)$
with index at most $d: = [{\rm Comm}(\Gamma): \Gamma]$.

Suppose that $d$ is finite. Then ${\rm Comm}(\Gamma)$ is cofinite
and $d$ equals $\Vol(X)$ divided by the covolume of ${\rm Comm}(\Gamma)$. 
By a result of Gelander  \cite{Gelander11},
there is a constant $K_n>0$ such that for each discrete subgroup 
$\Gamma' \subset \Isom(\Hbf^n)$ the number of generators 
of $\Gamma'$ is at most $K_n$  times the covolume of $\Gamma'$. 
Hence, since the covolume of ${\rm Comm}(\Gamma)$ at most $\Vol(X)$,
the number, $k$, of generators of  ${\rm Comm}(\Gamma)$ is at most $K_n \cdot \Vol(X)$.
Also, since $k \geq 1$, the covolume of ${\rm Comm}(\Gamma)$ is at least $1/K_n$. 
It follows that $d$ is at most $K_n \cdot \Vol(X)$.
The number of subgroups of index at most $d$  in a group generated by $k$ 
elements is at most the number of subgroups 
of index at most $d$ in the free group on $k$ generators. 
Therefore we have a bound on $N^*(X)$ in terms of $\Vol(X)$ when $d$ is finite.

If $d$ is not finite then, by a well-known theorem of Margulis \cite{Maclachlan-Reid},  
the lattice $\Gamma$ is arithmetic.  By another well-known theorem of Margulis  \cite{Maclachlan-Reid},
there are only finitely many hyperbolic manifolds,  $X_1, X_2, \ldots, X_k$, with volume less than $V$
which are quotients of arithmetic lattices. Let $\Gamma_1, \Gamma_2,\ldots, \Gamma_k$ 
be lattices such that $X_i = \Hbf^n/ \Gamma_i$ for $i=1, \ldots, k$. 
By \cite{Wang89}, the number, $n_i$, of discrete groups in which $\Gamma_i$ is contained 
is finite. Hence $N^*(X) \leq \max\{n_1, \ldots, n_k\}$ for $X$ arithmetic with volume
at most $V$. 
\end{proof}

\begin{remk}
If $ n \geq 4$, then the proof above can be simplified by use of a well-known result of  Wang:
If $n \geq 4$, then there are only finitely many hyperbolic manifolds
with volume less than $V$  (see Theorem 8.1 in  \cite{Wang72}). 
\end{remk}

Next, we bound the number of distinct $M$ such that $\Gamma(M)$ equals a given discrete group $\Gamma^*$
when $n=2$. By combining this with Theorem  \ref{thm:discrete-finite}, we will have Theorem \ref{thm:surface}.

\begin{prop}
Let $X= \Hbf^2/\Gamma$ be a complete finite area hyperbolic surface,
and let $\Gamma^*$ be a cofinite subgroup of $\Isom(\Hbf^2)$ that contains $\Gamma$.
The number of closed geodesics $M \subset X$ such that $\Gamma(M) = \Gamma^*$ is
at most $\frac{173}{4 \pi} \cdot Area(X)$. 
\end{prop}

\begin{proof}
Let $G$ be the union of the $k$ closed geodesics 
$M \subset X$ such that $\Gamma(M) \subset \Gamma^*$.  We have 
$\chi(X)~ =~  \chi(G)~ +~ \chi(X \setminus G)$ and hence by Lemma \ref{lem:eulerbound}
below, we have 
\[ k~ \leq~ p \, -\, \chi(G)~ \leq~ p\, +\, \chi(X \setminus G)\, -\, \chi(X) \] 
where $p$ is the number of curves needed to divide $X$ into pairs of pants. 
Since $X$ is a finite area hyperbolic surface, 
we have $p\leq-(3/2) \cdot \chi(X)$ and $-\chi(X)= \area(X)/ 2\pi$.

The Euler characteristic of $X \setminus G$ is at most the number, $j$, 
of connected components of
$X \setminus G$ with positive Euler characteristic. A component $C$ with $\chi(C)>0$
is homeomorphic to a disc, and hence it lifts to a compact polygon $P$
in $\Hbf^2$ bounded by geodesics $\tM$ such that $r_{\tM} \in \Gamma^*$. In particular, the 
group generated by reflections in the sides of $P$ is a discrete subgroup of $\Isom(\Hbf^2)$
and hence $\area(C) = \area(P) \geq \pi/42$ \cite{Siegel}.  The sum of the areas
of the $j$ components is at most $\area(X)$, and so we have $j \leq \area(X)/(\pi/42)$. 
\end{proof}

For the sake of completeness, we include the following lemma which is surely well-known. 
For a surface $X$, let $p(X)$ denote the maximal cardinality of a
collection of disjoint, nonparallel, essential, nonperipheral simple
closed curves on $X$. In other words, $p(X)$ is the number of curves 
needed to divide $X$ into pairs of pants. If $X$ is an $n$ times punctured
genus $g$ surface, then $p(X)= 3g-3+n$.

\begin{lem} \label{lem:eulerbound}
Let $X$ be a topological surface with $\chi(X)<0$. 
If $G \subset X$ is the union of $k$ 
mutually non-homotopic, essential, nonperipheral, closed curves that meet 
transversally, then 
\begin{equation} \label{eqn:euler}
    \chi(G)~ \leq~ -k~ +~  p(X).  
\end{equation}
Moreover, if equality holds then a subcollection of $G$ divides $X$ 
into pairs of pants. 
\end{lem}

\begin{proof}
If a graph $G$ and a curve $\gamma$ intersect transversally, then 
\begin{equation} \label{eqn:jump-down}
 \chi(G \cup \gamma)~ \leq~ \chi(G)~ -~ |G \cap \gamma|. 
\end{equation}
Indeed, using an orientation of $\gamma$, 
one can map the new edges that lie in $\gamma$ onto 
the set of intersection points $G \cap \gamma$. In particular, 
to each intersection point 
that was not a vertex in $G$, there is at least one new edge. 
If an intersection point $x$ is not a vertex of $G$, 
then $x$ bisects an edge of $G$. Hence for each new vertex, there are 
at least two new edges. Therefore, the Euler characteristic decreases by at 
least $|G \cap \gamma|$. 

We prove the claim by induction on the number of connected components of $G$. 
If $G$ is connected, then an inductive argument using (\ref{eqn:jump-down})
gives $\chi(G) \leq -k + 1$. Thus, if $p(X)\geq 1$, we have (\ref{eqn:euler})
and if equality holds, then $p(X)=1$ and hence $X$ is homeomorphic 
to either a once punctured torus $X_{1,0}$ or a four times punctured sphere $X_{0,4}$. 
A closed curve with a self-intersection point has negative Euler characteristic,
and two nonhomotopic simple closed curves on either $X_{1,0}$ or $X_{0,4}$ 
must intersect at least once.
Thus, by using induction on $k$, one finds that the equality implies that each closed 
curve in $G$ is simple. Any essential nonperipheral 
simple closed curve on $X_{1,0}$ or $X_{0,4}$ divides the surface into pairs of pants.

If $p(X)=0$, then $X$ is homeomorphic to 
a pair of pants. In this case, each essential nonperipheral closed curve $\alpha$
has at least one self-intersection and thus $\chi(\alpha) \leq -1$. 
An induction argument using (\ref{eqn:jump-down}) then shows $\chi(G) \leq -k$.  

Suppose that the claim is proven when the graph has $j$ components,
and suppose that $G$ is a graph with $j+1$ components. In other words, 
$G= G_+ \sqcup G_-$ where $G_+$ is connected and $G_-$ has $j$ components. If $G_+$ 
consists of a simple closed curve, then each connected component of $G_-$ lies in a connected component $X_-$
of $X \setminus G_+$. By the inductive hypothesis, we have 
$\chi(G)= \chi(G_-) \leq -(k-1) +p(X_-) = -k + p(X)$. 
If equality holds, then by the inductive hypothesis a subset $G'$ of $G_-$
divides $X_-$ into pairs of pants, and therefore $G_+\sqcup G'$ divides
$X$ into pairs of pants. 

If $G_+$ is not a simple closed curve, then there exist disjoint 
subsurfaces $X_+$ and $X_-$ so that $G_{\pm} \subset X_{\pm}$ and $\chi(X_{\pm})<0$.
Thus, by the inductive hypothesis, 
$\chi(G)= \chi(G_+)+ \chi(G_-) \leq -k + p(X_+)+ p(X_-) \leq -k + p(X)$.  
Note that  $p(X_+)+ p(X_-) <  p(X)$ and so equality cannot hold in this case.
\end{proof}


\section{Some constructions}  \label{ex:figure-eight}

Not every totally geodesic component of a nodal set is the fixed
point set of a reflection isometry. To see this, we construct here a closed 
hyperbolic surface $X$ of genus two and an eigenfunction on $X$ 
such that vanishes on a closed geodesic with a point of self-intersection. 
Because the self-intersection is transversal, it can not be the fixed
point set of a reflection.  
 
Let $O$ be a regular geodesic octagon in the hyperbolic plane such that 
the angle at each vertex is $\pi/2$ \cite{Buser}. See the figure below.

\begin{center}
\includegraphics[scale=.4]{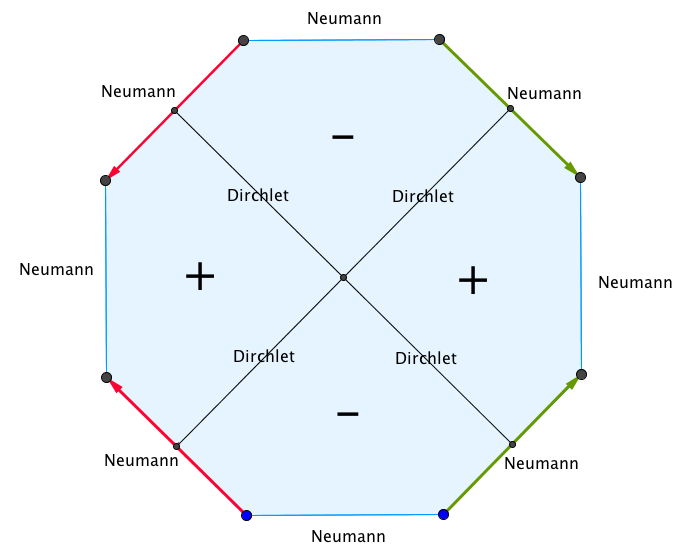}
\end{center}

Choose two geodesics that intersect orthogonally in the center of the octagon 
and that bisect opposing sides. These geodesics divide the octogon into four right-angled
pentagons. For one of these pentagons solve the eigenvalue problem 
with Dirichlet conditions on the interior geodesics and Neumann conditions
on the exterior geodesics. Then extend to the other pentagons via 
Schwarz reflection. Because of the Neumann conditions on the exterior sides, 
we may identifiy the oriented red sides together and the oriented green sides to obtain 
an eigenfunction on a pair of pants such that the eigenfunction vanishes 
on a closed geodesic
with one intersection point. By doubling the surface, we obtain an 
eigenfunction on a genus two surface with the desired properties. 

One can also construct hyperbolic surfaces with simple closed geodesics
that lie in the nodal set of an eigenfunction but that do not belong to
the fixed point set of a reflective isometry. For example, let $P$ be a hyperbolic 
pair of pants with equal boundary lengths \cite{Buser}. Solve the eigenvalue 
problem for the Laplacian on $P$ with Dirichlet conditions on one boundary component
and Neumann boundary conditions on the other two boundary components. 
Glue four copies of $P$ together as indicated in the figure below.

\includegraphics[scale=.3]{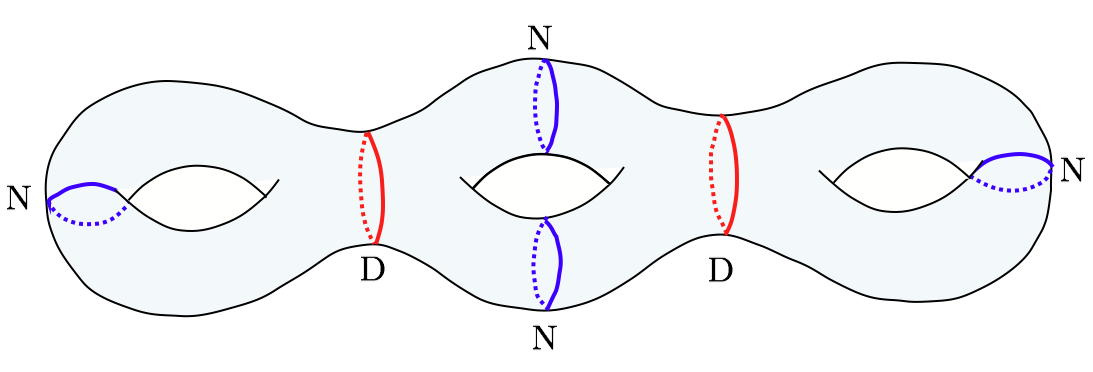}

\noindent 
In particular, we use Schwarz reflection to extend the eigenfunction on $P$
to the entire genus three surface. The eigenfunction vanishes on the simple 
geodesics colored red in the figure above. 

Although in the examples above the closed geodesic is not the fixed point set 
of a reflection isometry, each surface does have a finite covering 
such that the lift of the geodesic is the fixed point set of a reflection.
More generally, we have the following.

\begin{prop} \label{prop:cover}
Let $X$ be a hyperbolic manifold with finite volume. 
If $f$ is an eigenfunction that vanishes on an immersed totally
geodesic hypersurface $M \subset X$, then there exists a finite cover $p:Y \to X$
and a lift $N$ of $M$ such that $N$ is the fixed point set of a reflection
isometry and the lift of $f$ to $Y$ vanishes on $N$. 
\end{prop}

\begin{proof}
Let $\Gamma$ be a lattice such that $X = \Hbf^n/\Gamma$, and let $\tf$
be the lift of $f$ that is invariant under $\Gamma$. We have   
$\Gamma \subset \langle \Gamma, r \rangle \subset \Gamma_{\tf}$
where $r$ is the reflection through a lift $\tM$ of $M$ to $\Hbf^n$.  
Since $\Gamma$ is cofinite, the index $[ \langle \Gamma, r \rangle \, : \, \Gamma]$ is finite.
Hence the intersection $\Gamma' = \Gamma \cap (r \Gamma r^{-1})$ has 
finite index in $\langle \Gamma, r \rangle$. We have $r \Gamma' r^{-1} = \Gamma'$
and hence $r$ descends to a reflection isometry on $\Hbf^n/ \Gamma'$ whose
fixed point set is $\tM/\Gamma'$. Since $\tf$ vanishes on $\tM$, the lift of
$f$ to $\Hbf^n/\Gamma'$ vanishes on $\tM/\Gamma'$.  
\end{proof}

\section{Some questions} \label{section:questions}

In the introduction we observed that if the isometry group of a compact Riemannian 
manifold contains a one parameter subgroup $K$, then $N(X)$ may be infinite. 

\begin{quest} \label{quest:finite}
If the isometry group of a compact Riemannian manifold 
is discrete, then is $N(X)$ finite?  
\end{quest}

Recently, in \cite{EP15} it was proven that given $n \in \Nbb$ and 
an isotopy class $C$ of a hypersurface in a compact topological manifold $M$
one can construct a Riemannian metric $g$ on $M$ and an eigenfunction $\phi$
of the Laplacian $\Delta_g$ such that $\phi$ vanishes on $n$ disjoint 
hypersurfaces each belonging to the homotopy class $C$. The methods of  \cite{EP15}
can be applied to solve the following problem:  Given a compact manifold $M$,
$\epsilon>0$ and $n \in \Nbb$, there exists a Riemannian metric $g$ on $M$,
$n$ disjoint hypersurfaces $H_i$, and an eigenfuntion $\phi$ such that 
each $H_i$ is at distance at most $\epsilon$ from a totally geodesic 
hypersurface and $\phi|_{H_i}=0$. These examples do not have negative curvature.

In \S \ref{section:surface}, we showed that if the dimension, $n$, of $X$ equals 2,
then one can bound $N'(X)$ and hence $N(X)$ 
in terms of the area of $X$. For $n \ge 4$, Theorem \ref{thm:main} 
together with Wang finiteness \cite{Wang72} implies that there exists a bound on 
$N(X)$ in terms of the volume of $X$.\footnote{Thanks to the anaonymous 
referee for reminding us to state this fact.}

\begin{quest} \label{quest:bound}
If $n = 3$, can one bound $N(X)$ in terms of the volume of $X$?
\end{quest}
In \S \ref{section:surface}, we reduced Question \ref{quest:bound} to a
quantification of Proposition \ref{prop:finite-realized} for $n \geq 3$. 
A related question: To what extent does $\Gamma(M)$ determine the totally geodesic 
hypersurface $M$?

For each $g$ and $p$, the supremum, $N_{g,p}$, of $N(X)$  over the finite
area hyperbolic structures on a genus $g$ surface with $p$ punctures
is finite.  One can easily construct examples that show that 
$N_{g,p}$ is unbounded.

\begin{quest} \label{conj:bounded}
How does $N_{g,p}$ grow as $g$ and/or $p$ tend to infinity?
\end{quest}

Finally, we wonder how general Proposition \ref{prop:cover} is.  

\begin{quest}
If a totally geodesic hypersurface $H$ belongs to the nodal set of an eigenfunction on 
a (real-analytic) 
Riemanian manifold $(M,g)$, then does there exist a Riemannian covering $(\tM, \tg)$
and a reflection isometry of $(\tM, \tg)$ whose fixed point set contains a lift of $H$?  
\end{quest}


\end{document}